\documentclass[11pt,a4paper]{article}
\usepackage[utf8]{inputenc}
\usepackage{amsmath, amssymb, amsthm, comment, enumerate, float, graphicx, hyperref, geometry, a4wide, xcolor, verbatim, thm-restate}
\usepackage[capitalise]{cleveref}

\newtheorem{corollary}[subsection]{Corollary}
\newtheorem{lemma}[subsection]{Lemma}
\newtheorem{proposition}[subsection]{Proposition}

\theoremstyle{remark}
\newtheorem*{remark}{Remark}

\newcommand{\mailto}[1]{\href{mailto: #1}{#1}}
\newcommand{\cP}{\mathcal{P}}

\DeclareMathOperator{\sat}{sat}

\title{
On the number of edges in saturated partial embeddings of maximal planar graphs
}
\author{János Barát\thanks{HUN-REN Alfr\'ed R\'enyi Institute of Mathematics, Budapest, Hungary and University of Pannonia, Veszprém, Hungary, \mailto{barat@renyi.hu}
}
\and
Zoltán L. Blázsik\thanks{Bolyai Institute, University of Szeged, Aradi v\'ertan\'uk tere 1, 6720
Szeged, Hungary. \\ University of Johannesburg Auckland Park, 2006 South Africa, \mailto{blazsik@server.math.u-szeged.hu}}
\and
Balázs Keszegh\thanks{HUN-REN Alfr\'ed R\'enyi Institute of Mathematics and ELTE Eötvös Loránd University, Budapest, Hungary. 
~\mailto{keszegh@renyi.hu}
}
\and
Zeyu Zheng\thanks{Department of Mathematical Sciences, Carnegie Mellon University, Pittsburgh, PA 15213, USA. \mailto{zeyuzhen@andrew.cmu.edu}
}
}

\date{}
\begin{document}
\maketitle
\begin{abstract}
We investigate the extremal properties of saturated partial plane embeddings of maximal planar graphs. For a planar graph $G$, the plane-saturation number $\sat_{\mathcal{P}}(G)$ denotes the minimum number of edges in a plane subgraph of $G$ such that the addition of any edge either violates planarity or results in a graph that is not a subgraph of $G$. We focus on maximal planar graphs and establish an upper bound on $\sat_{\mathcal{P}}(G)$ by showing there exists a universal constant $\epsilon > 0$ such that $\sat_{\mathcal{P}}(G) < (3-\epsilon)v(G)$ for any maximal planar graph $G$ with $v(G) \geq 16$. This answers a question posed by Clifton and Simon. 
Additionally, we derive lower bound results and demonstrate that for maximal planar graphs with sufficiently large number of vertices, the minimum ratio \(\sat_{\mathcal{P}}(G)/e(G)\) lies within the interval \((1/16, 1/9 + o(1)]\).
\end{abstract}

\section{Introduction}

Unless explicitly stated otherwise, all graphs considered in this paper are finite, undirected and simple. Recall from fundamental graph theory that a \emph{planar graph} is a graph that admits a plane embedding, while a \emph{plane graph} refers to a specific embedding of a planar graph in the plane. For simplicity, we often treat a plane graph as both a graph and its corresponding embedding, depending on the context. We denote the vertex set and the edge set of a graph $G$ by $V(G)$ and $E(G)$, respectively, and the number of vertices and edges by $v(G)$ and $e(G)$.

The degree of a vertex $v$ is denoted by $d(v)$, and the number of degree $i$ neighbors of $v$ is denoted by $d_i(v)$. Vertices of degree $3$ play an important role in our subsequent discussion. Their count in $G$ is denoted by $n_3(G)$.

A classic problem in extremal graph theory is the graph saturation problem. Given a graph $H$, one seeks to determine the minimum number of edges in a graph $G$ such that $G$ does not contain $H$ as a subgraph, but the addition of any edge to $G$ (on the same vertex set) results in a copy of $H$ as a subgraph. For a more in-depth overview of the classic saturation problem, we refer the readers to the survey \cite{faudree2011survey}.

The study of the plane saturation problem, initiated by Clifton and Salia \cite{clifton2024saturated}, is a variant of the saturation problems. A plane graph $H$ is a \emph{plane-saturated subgraph} of $G$, if adding an extra edge to $H$ either violates the planarity of the embedding, or results in a graph that is not a subgraph of $G$. We define the \emph{plane-saturation number}, $\sat_\cP(G)$, as the minimum value of $e(H)$, where $H$ is a plane-saturated subgraph of $G$.

Very recently, parallel to our studies, Clifton and Simon investigated the plane saturation problem for labeled and unlabeled maximal planar graphs in \cite{clifton2024saturated2}. They demonstrated through construction that there exists an infinite family of maximal planar graphs satisfying $\sat_\cP(G)\ge 3v(G)/2-3$. They also asked whether there exists some $\epsilon$ such that $\sat_\cP(G)<(3-\epsilon)v(G)$ for any maximal planar graph $G$. The only previously known result was independently noticed by us and the authors of \cite{clifton2024saturated2}.

\begin{proposition}\label{prop:linearGap}
Let $G$ be a maximal planar graph on $n$ vertices, with degree sequence $d_1 \leq d_2 \leq \cdots \leq d_n$. Suppose there exist an index $k \in [n-1]$ and a constant $c > 0$ such that $d_{k+1} - d_k \geq cn$. Then, 
\[
\sat_\cP(G) \leq (3-c)n - 2.
\]
\end{proposition}

In this paper, we address this question and provide an affirmative answer to the general case.

\begin{restatable}{theorem}{thmMain}\label{thm:main}
There exists a universal constant $\epsilon>0$ such that $\sat_\cP(G)<(3-\epsilon)v(G)$ holds for any maximal planar graph $G$ on at least $16$ vertices.
\end{restatable}

For lower bound results, Clifton and Salia \cite{clifton2024saturated} showed that any planar graph $G$ without degree~$1$ or degree~$2$ twins satisfies $\sat_\cP(G)>e(G)/16$, where twins refer to a pair of vertices that share the same neighborhood. Since maximal planar graphs do not contain degree~$1$ or degree~$2$ twins, the following corollary immediately holds.
\begin{corollary}\label{coro:lowerBound}
    For any maximal planar graph $G$, $\sat_\cP(G)>(3v(G)-6)/16$.
\end{corollary}

Clifton and Salia also conjectured in \cite{clifton2024saturated} that for planar graphs with minimal degree $3$, the ratio $\sat_\cP(G)/e(G)$ is bounded below by $1/9$. We present a construction of an infinite family of maximal planar graphs that achieves this ratio asymptotically, as detailed in the following theorem.

\begin{restatable}{theorem}{thmConstruction}\label{thm:construction}
    There exists an infinite family of maximal planar graphs such that each graph $G$ in the family satisfies $\sat_\cP(G)\le \frac{v(G)+82}{3}$.
\end{restatable}

\section{Propositions of maximal planar graphs}

In this section, we recall and establish some basic results about maximal planar graphs that are useful in our proofs.

Note that in a maximal planar graph with at least $4$ vertices, the minimum degree is at least $3$. The following propositions show that degree $3$ vertices can not be too dense.
\begin{proposition}\label{prop:main}
    In a maximal planar graph $G$ with at least $5$ vertices, the following properties hold: \begin{enumerate}[(i).]
        \item The set of degree $3$ vertices forms an independent set.\label{item:independent}
        \item Every vertex $v$ of degree at least $4$ has at most $\lfloor d(v)/2\rfloor$ many degree $3$ neighbors.\label{item:neighbors}
        \item Every vertex $v$ of degree at least $4$ has at least three neighbors of degree at least $4$, unless $G$ is $K_5^-$ (the complete graph on 5 vertices with one edge deleted).
        \label{item:highNeighbors}
        \item Every vertex $v$ of degree at least $4$ has two neighbors that are adjacent in $G$, both of which have degree at least $4$, as illustrated in Figure~\ref{fig:adj}.
        \label{item:adjacentNeighbors}
    \end{enumerate}
    \begin{figure}[H]
    \centering
    \includegraphics[width=0.3\linewidth]{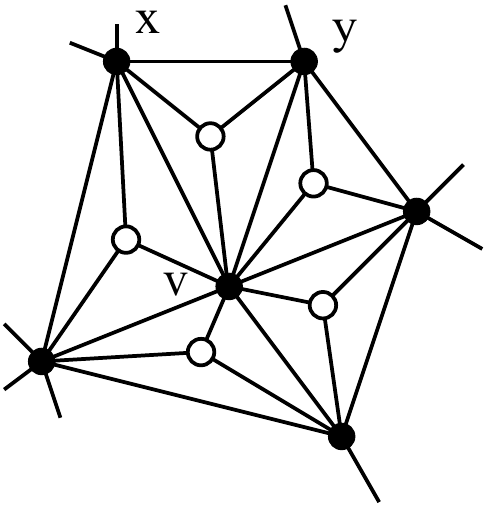}
    \caption{Vertices $x$ and $y$ are adjacent neighbors of $v$.}
    \label{fig:adj}
    \end{figure}
\end{proposition}
\begin{proof}~
    \begin{enumerate}[$(i)$.]
        \item Fix a plane embedding of $G$ and suppose to the contrary  that two degree $3$ vertices $u,v$ are adjacent. Since $G$ is a maximal planar graph and thus a triangulation, there exist vertices $x$ and $y$ such that $uvx$ and $uvy$ are triangular faces. Since both $u$ and $v$ have degree exactly $3$, these are all the edges incident to $u$ and $v$. Consequently, there must be a face that contains both paths $xuy$ and $xvy$. This implies the existence of parallel edges between $x$ and $y$, contradicting the fact that $G$ is simple.
        \item Since $G$ is a maximal planar graph, the neighborhood of $v$ induces a cycle in the plane embedding. By ($\ref{item:independent}$), the degree $3$ vertices in this neighborhood form an independent set. Therefore, in the cyclic ordering of neighbors around $v$, no two degree $3$ vertices can be consecutive. Hence, at most $\lfloor d(v)/2 \rfloor$ vertices in the neighborhood can have degree $3$.
        \item For vertices of degree at least $5$, this follows immediately from ($\ref{item:neighbors}$). For the case where $d(v) = 4$, suppose to the contrary that $v$ has fewer than three neighbors of degree at least $4$, and thus at least two degree $3$ neighbors. Let $x$ and $y$ be these degree $3$ neighbors, and let $a$ and $b$ be the other two neighbors of $v$. Since both $x$ and $y$ have degree exactly $3$, their closed neighborhoods must each induce a $K_4$ containing the edge $ab$. This forces either the existence of parallel edges between $a$ and $b$ (contradicting simplicity), or that the graph is precisely $K_5^-$ (the excluded case).
        \item Fix a plane embedding of $G$ and consider any vertex $v$ of degree at least $4$. Suppose to the contrary that any two neighbors of $v$ of degree at least $4$ are non-adjacent in $G$. The neighborhood of $v$ induces a cycle in this embedding and the vertices of degree $3$ must alternate with vertices of higher degrees in the cycle. Consider any two vertices $a,b$ of degree at least $4$ that share a degree $3$ neighbor $x$ in this cycle. Since $x$ has degree exactly $3$, $axb$ must form a triangular face in the embedding. This implies that $a$ and $b$ are adjacent, which contradicts our assumption.
    \end{enumerate}
\end{proof}

The following proposition gives an upper bound on the number of degree $3$ vertices in $G$.
\begin{proposition}\label{prop:max3}
    For $n\ge 5$, a maximal planar graph $G$ on $n$ vertices has at most $\frac{2n-4}{3}$ degree $3$ vertices. This bound is sharp.
\end{proposition}
\begin{proof}
By Euler's formula, the number of faces in $G$ is $2n-4$. From \cref{prop:main}(\ref{item:independent}), it follows that each face of $G$ contains at most one degree~$3$ vertex. Since each degree~$3$ vertex is incident to exactly $3$ faces, the total number of degree $3$ vertices is at most $\frac{2n-4}{3}$.

To see that this bound is attained, begin with any maximal planar graph and, for each triangular face, add a new vertex of degree~$3$ adjacent to the three vertices of that face. This construction ensures that each face contains exactly one degree~$3$ vertex, thus achieving the claimed bound.
\end{proof}

\section{The upper bound}

In this section, we prove \cref{thm:main}. We consider two distinct regimes in our proof, which together cover all possible cases.

\begin{lemma}\label{lemma1}
    Let $G$ be a maximal planar graph on $v(G)=n$ vertices with $n_3(G)=k$, and $k\le \lfloor n/2\rfloor -\epsilon n-1$ for some $\epsilon>0$. Then, $G$ has a plane-saturated subgraph $H$ with at most $(3-\epsilon/2)n$ edges.
\end{lemma}
\begin{proof}
Let $\phi$ be a fixed plane embedding of $G$. We shall construct $H$ through an iterative process:
\begin{enumerate}
    \item Initialize $H_0$ with an arbitrary face of $\phi(G)$.
    \item For $i \geq 1$, construct $H_i$ from $H_{i-1}$ by adding a triangular face of $\phi(G)$ that shares at least one edge with $H_{i-1}$. This process adds at most one new vertex per step.
    \item Continue this process until exactly $\lceil n/2 \rceil {-}2$ vertices have been added to $H_0$ to form $H_t$.
\end{enumerate}

At this stage, $H_t$ contains exactly $\lceil n/2 \rceil + 1$ vertices. The remaining $\lfloor n/2 \rfloor - 1$ vertices of $G$ are yet to be added. Note that $H_t$ contains at least $\lceil n/2 \rceil-1$ triangular faces, as each vertex after the initial three contributes at least one new triangular face. Now, place each of the remaining vertices inside a distinct triangular face of $H_t$ and add edges to obtain a maximal plane-saturated subgraph $H$. By construction, each of these $\lfloor n/2 \rfloor - 1$ newly added vertices has degree at most $3$ in $H$.

In the original graph $G$, $(n-k)$ vertices have degree at least $4$. However, in $H$, at least $\lfloor n/2 \rfloor - 1$ vertices have degree $3$. Since $k < \lfloor n/2\rfloor-1$, this means that at least $\lfloor n/2 \rfloor - 1 - k$ vertices must have lower degree in $H$ than in $G$. Therefore, $H$ has at most \[3n-6-\frac{1}{2}\left(\left\lfloor \frac{n}{2} \right\rfloor - 1 - k\right)<\left(3-\frac{\epsilon}{2}\right) n\] many edges.
\end{proof}

\begin{lemma}\label{lem:many3}
If $G$ is a maximal planar graph on $n$ vertices with $n_3(G)=k$, and $k\ge (0.4+\epsilon)n$ for some $\epsilon>0$, then $G$ has a plane-saturated subgraph $H$ with at most $(3-\frac{5\epsilon}{3})n$ edges.
\end{lemma}

\begin{proof}
Fix a plane embedding $\phi$ of $G$. Let $U$ denote the set of vertices of degree at least~$4$ in $G$, and let $H_U$ be the plane subgraph of $G$ induced by $U$ under the embedding $\phi(G)$. We claim that $H_U$ is a maximal planar graph. To see why, observe first that removing a vertex of degree $3$ from a maximal planar graph results in another, smaller maximal planar graph. Indeed, the removal of a vertex of degree $3$ and its incident edges keeps all faces triangular, thereby preserving maximal planarity. Next, note that $H_U$ is obtained from $\phi(G)$ by repeatedly removing vertices of degree $3$ in $V(G)\setminus U$. Since at each step the resulting graph remains maximal planar, it follows that the final subgraph $H_U$ is also a maximal planar graph.

We extend $H_U$ to a plane subgraph $\widetilde{H_U}$ of $G$, where each vertex in $U$ has degree at least~$4$, by adding edges incident to $U$. Since these additional edges must be incident to degree $3$ vertices of $G$, this process incorporates some degree $3$ vertices into our construction. By \cref{prop:main}($\ref{item:highNeighbors}$), each vertex in $U$ requires at most one additional edge to achieve degree~$4$ in $\widetilde{H_U}$. Since $H_U$ is maximal planar on $n-k$ vertices, by \cref{prop:max3}, $n_3(H_U)< 2v(H_U)/3 = 2(n-k)/3$. Therefore, we introduced at most $2(n-k)/3$ additional vertices in this process. Let $W$ be the set of remaining degree $3$ vertices that were not used to construct $\widetilde{H_U}$. Since we started with $k$ degree $3$ vertices and used at most $2(n-k)/3$ of them, we still have \[|W| \geq k-\frac{2(n-k)}{3} = \frac{5k-2n}{3}>0\] degree 3 vertices left.

Consider a trianglular face of the graph $\widetilde{H_U}$ and  denote it by $T$. We construct $H_W$ by placing all vertices of $W$ inside $T$. In any plane extension of $H_W$, at most one vertex in $W$ can be adjacent to all three vertices of $T$ without violating planarity. Therefore, each of the remaining vertices in $W$ must miss at least one edge to the vertices of $T$. Since $|W| \ge  {(5k-2n)}/{3}$, any plane extension of $H_W$ must have fewer edges than $G$ at least by $|W|-1\ge {(5k-2n-3)}/{3}$. Therefore,
\[ \sat_\cP(G) \le e(G) - \frac{5k-2n-3}{3} = (3n-6) - \frac{5k-2n-3}{3} < \left (3-\frac53\epsilon\right )n,\] where the last inequality follows from our assumption that $k \ge (0.4+\epsilon)n$.
\end{proof}

Putting the two regimes together, we have:
\thmMain* 
\begin{proof}
    For $v(G)=n\ge 16$, observe that \[
    \left(0.4+\frac{1}{150}\right)n < \lfloor 0.5n\rfloor - \frac{1}{500}n-1
    \]
    Hence, in a maximal planar graph $G$ with at least $16$ vertices, we must either have $n_3(G)<\lfloor 0.5 n\rfloor - \frac{1}{500}n-1$, or $n_3(G)\ge(0.4+\frac{1}{150})n$. Depending on the number of degree $3$ vertices, we can use either \cref{lemma1} or \cref{lem:many3} to deduce that there always exists a plane-saturated subgraph of $G$ with at most $(3-\frac{1}{300})v(G)$ edges, thus the theorem holds with $\epsilon=1/300$.
\end{proof}
\begin{remark}
    The strength of the bound in the proof above depends on the smallest number of vertices it applies to. In particular, for sufficiently large $t$, we have the bound $\sat_\cP(G)<(3-\epsilon)v(G)$ for all maximal planar graphs with more than $t$ vertices, where $\epsilon$ can be set arbitrarily close to $1/26$.
\end{remark}

\section{Lower bound results}

Although the more general result in \cite{clifton2024saturated} implies Corollary \ref{coro:lowerBound}, their proof is highly technical and lengthy. In this section, we first provide a significantly simpler proof for a slightly weaker lower bound on maximal planar graphs.

\begin{proposition}
For any maximal planar graph $G$, $\sat_\cP(G)\ge \frac{v(G)+4}{6}$.
\end{proposition}
\begin{proof}
Let $G$ be a maximal planar graph on $v(G)=n$ vertices and let $H$ be a plane graph, which is a subgraph of $G$. We prove that for any such $H$ with $e(H)<\frac{n+4}{6}$, it is always possible to add a plane edge $e$ to $H$ such that $H+e$ remains a subgraph of $G$. 

Denote by $f: V(H) \to V(G)$ the mapping of the vertices of $H$ into $G$ as a subgraph. Let $I = \{v \in V(H) : d_H(v) = 0\}$ be the set of isolated vertices in $H$. By our assumption on $e(H)$:
$$|I| \geq n - 2e(H) > n - 2\cdot\frac{n+4}{6} = \frac{2n-4}{3}.$$

Fix a planar embedding $\phi$ of $G$. Let $H_G$ be the induced plane subgraph of $\phi(G)$ on vertex set $\phi(V(G) \setminus f(I))$. Observe that $H - I$ is a plane subgraph of $H_G$. 
Since $v(H_G)=n-|I| <\frac{n+4}{3}$, the number of faces in $H_G$ is less than $2\cdot\frac{n+4}{3} - 4 = \frac{2n-4}{3}$. By the pigeonhole principle, there exist embedded vertices $u,v \in \phi(f(I))$ that lie in the same face $F$ of $H_G$. Since $G$ is maximal planar, we can further require $uv \in \phi(E(G))$.

Similarly, since $H$ has at most $\frac{2n-4}{3}$ faces, there exist vertices $x,y \in I$ that lie in the same face of $H$. Define a modified mapping $\tilde{f}$ by setting $\tilde{f}(x) = u$ and $\tilde{f}(y) = v$ and $\tilde{f}=f$ otherwise. Now $\phi(\tilde{f})$ shows that $H + xy$ is a plane subgraph of $\phi(G)$ or $G$.
\end{proof}

To prove \cref{thm:construction}, we describe a construction of an infinite family of graphs. The construction is inspired by Construction 1 in \cite{clifton2024saturated}.

\thmConstruction*
\begin{proof}
We construct an infinite family of maximal planar graphs $\{G_k\}_{k\geq 6}$ with corresponding plane-saturated subgraphs $\{H_k\}_{k\ge 6}$ that satisfy the desired property.

For each integer $k\ge 6$, we first construct $G_k'$ as follows. Begin with $C_k$, a cycle on $k$ vertices. Add two vertices $v_1$ and $v_2$, and connect each of them to all vertices of $C_k$, creating $2k$ new edges. At this stage, the graph has $k+2$ vertices and $3k$ edges, forming $2k$ triangular faces. For each triangular face $f$, we add a new vertex and connect it to all three vertices of $f$ to get $G_k'$. This construction ensures that $G_k'$ is maximal planar.

Consider a plane embedding of the disjoint union of $G_k'$ and a triangle $x,y,z$. In the plane embedding of $G_k'$, one vertex $u$ of the face that contains $xyz$ has degree $3$, which is obtained from the final step when constructing $G_k'$. We set the neighbors of  $u$ to be $v_2, a, b$.
Connect $u$ to all three vertices of $x,y,z$, and add the additional edges $xa,za,zb$ to make the resulting graph maximal planar. Denote this graph by $G_k$. Notice that $v_2$ is adjacent to $u$, but not adjacent to any of $x,y,z$. See the left-hand side of Figure~\ref{fig:tri}. The number of vertices in $G_k$ is given by \[
v(G_k)=v(G_k')+v(K_3)=k+2+2k+3=3k+5.
\]

Next, we define a plane subgraph $H_k$ of $G_k$. Start with a plane embedding of $K_4$, which creates three bounded triangular faces $F_1,F_2,F_3$. In $F_1$, we embed a cycle $C_k$. In $F_2$, we embed two copies of $K_{1,11}$ (stars with a central vertex of degree $11$) and connect one center, $v_2'$ say, to one vertex of $F_2$. In $F_3$, we place $2k-23$ isolated vertices. See the right-hand side of Figure~\ref{fig:tri}.

    \begin{figure}[H]
    \centering
    \includegraphics[width=0.95\linewidth]{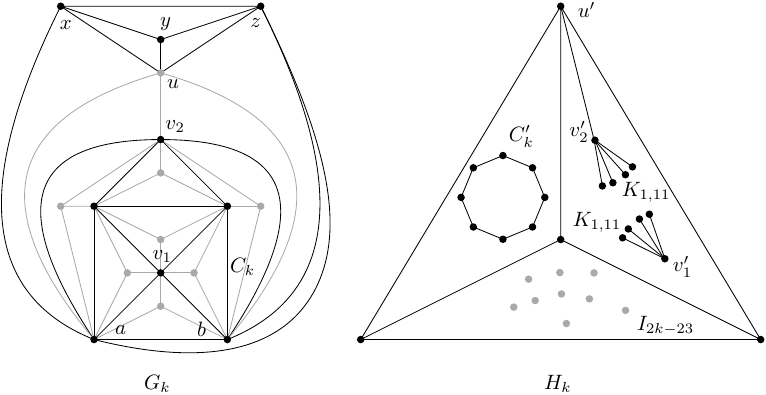}
    \caption{$G_k$ and the saturated plane drawing of its subgraph $H_k$.}
    \label{fig:tri}
    \end{figure}

We shall demonstrate that it is impossible to add an edge to $H_k$ without destroying planarity such that the resulting plane graph is still a subgraph of $G_k$.

Observe that in $G_k$, all vertices have degree at most $10$, except for $v_1$ and $v_2$. As $H_k$ contains exactly two vertices of degree at least $11$ (the centers of the $K_{1,11}$'s), these must correspond to $v_1$ and $v_2$ in $G_k$.
In particular $v_2'$ must correspond to $v_2$. Once $v_1$ and $v_2$ are identified, the cycle $C_k$ in $H_k$ must correspond to the original $C_k$ in $G_k'$ (up to rotation and reflection), as it is the only remaining cycle of length $k$ in the graph. Now the remaining vertices in both $H_k$ and $G_k$ form a $K_4$. They must correspond to each other. In particular, the $K_4$ in $G_k$ is induced by $u,x,y,z$. 
The remaining isolated vertices in $H_k$ must correspond to the remaining vertices in $G_k$.
These are vertices added in bounded triangular faces of $G_k'$, so they form an independent set.
This unique correspondence, together with the prescribed planar embedding of $H_k$, ensures that any additional edge from $G_k$ would create a crossing when added to $H_k$.
Observe that any missing edge of $G_k$ would connect vertices of $H_k$ in different faces. Therefore, $H_k$ is indeed a plane-saturated subgraph of $G_k$. By construction, \[
e(H_k) = e(K_4) + e(C_k) + 2e(K_{1,11})+1 = 6 + k + 22 + 1 = k + 29,
\]
from which we deduce
\[
\sat_\cP(G_k) \leq k + 29 = \frac{v(G_k) - 5}{3} + 29 = \frac{1}{3}v(G_k) + \frac{82}{3},
\]
completing the proof.
\end{proof}

\section*{Acknowledgement}
The first author is partially supported by ERC Advanced Grant ``GeoScape" No.88271. and NKFIH Grant K.131529. The second author was supported by the ÚNKP-23-4-SZTE-628 New National Excellence Program during the first stages of this research work, and currently is supported by the EKÖP-24-4-SZTE-609 Program. Both of these programs belong to the Ministry for Culture and Innovation from the source of the National Research, Development and Innovation Fund. The third author is supported by the National Research, Development and Innovation Office -- NKFIH under the grant K 132696 and by the ERC Advanced Grant ``ERMiD''. This research has been implemented with the support provided by the Ministry of Innovation and Technology of Hungary from the National Research, Development and Innovation Fund, financed under the  ELTE TKP 2021-NKTA-62 funding scheme. The fourth author is supported in part by NSF grant DMS-2154063.

This research was initiated during the 15th Emléktábla Workshop, 2024. The authors are thankful for the organizers for inviting them. The authors also thank Dániel T. Nagy for fruitful early discussions on the project.

\bibliographystyle{plain}
\bibliography{citation}

\begin{thebibliography}{1}

\bibitem{clifton2024saturated}
Alexander Clifton and Nika Salia.
\newblock Saturated partial embeddings of planar graphs.
\newblock {\em arXiv preprint arXiv:2403.02458}, 2024.

\bibitem{clifton2024saturated2}
Alexander Clifton and D{\'a}niel~G Simon.
\newblock Saturated partial embeddings of maximal planar graphs.
\newblock {\em arXiv preprint arXiv:2412.06068}, 2024.

\bibitem{faudree2011survey}
Jill~R Faudree, Ralph~J Faudree, and John~R Schmitt.
\newblock A survey of minimum saturated graphs.
\newblock {\em The Electronic Journal of Combinatorics}, 1000:DS19--Jul, 2011.

\end{thebibliography}

\end{document}